\documentclass[12pt]{article}
\usepackage{setspace} 
\usepackage[dvips]{graphicx} 
\usepackage{amsmath,amsthm,amsfonts}
\usepackage[utf8]{inputenc}
\RequirePackage[colorlinks,citecolor=blue,urlcolor=blue,linkcolor=blue]{hyperref}
\hypersetup{
colorlinks = true,
citecolor=blue,
urlcolor=blue,
linkcolor=blue,
pdfauthor = {Minjian Yuan},
pdfkeywords = { Wronskian determinant, Darboux transformation, killed Brownian motion, transition probability, Gegenbauer polynomials},
pdftitle = {On Gegenbauer polynomials and Wronskian determinants of trigonometric functions},
pdfpagemode = UseNone
}

    \def\qed{\hfill$\sqcap\kern-8.0pt\hbox{$\sqcup$}$\\}
    
    \def\L{{\mathcal L}}

    \def\D{{\mathcal D}}

    \def\r{{\mathbb R}}
    \def\wr{{\textnormal{Wr}}}
    
    \def\d{{\textnormal d}}

	\newtheorem{theorem}{Theorem}
	\newtheorem{lemma}{Lemma}
	\newtheorem{proposition}{Proposition}

\title{On Gegenbauer polynomials and Wronskian determinants of trigonometric functions }
\author{ 
{Minjian Yuan \footnote{Dept. of Mathematics and Statistics,  York University,
4700 Keele Street, Toronto, ON, M3J 1P3, Canada.   Email: yuanm@yorku.ca}}}

\date{\today}

\begin{document}
\maketitle

\begin{abstract}
In \cite{Larsen_1990} M. E. Larsen evaluated the Wronskian determinant of functions $\{\sin(mx)\}_{1\le m \le n}$. We generalize this result and compute the Wronskian of $\{\sin(mx)\}_{1\le m \le n-1}\cup \{\sin((k+n)x\} $. We show that this determinant can be expressed in terms of Gegenbauer  orthogonal polynomials and we give two proofs of this result: a direct proof using recurrence relations and a less direct (but, possibly, more instructive)  proof based on Darboux-Crum transformations. 
\end{abstract}

{\vskip 0.15cm}
{\vskip 0.25cm}

\section{Introduction and the main result}\label{section_intro}

Determinants play a crucial role in various branches of mathematics. For example, the Vandermonde determinant is widely used in polynomial interpolation, whereas determinants of Toeplitz matrices feature prominently in random matrix theory and integrable systems. Consequently, there have been numerous studies on finding explicit formulas for specific determinants. In \cite{Houston_2024}, R.~Houston, A.\,P.~Goucher, and N.~Johnston present a new explicit formula for the determinant that contains superexponentially fewer terms than the usual Leibniz formula. In \cite{Sun_2024}, Z.~Sun evaluates a determinant involving quadratic residues modulo primes. Our inspiration for this work comes from the following explicit formula obtained by M.\,E.~Larsen in \cite{Larsen_1990}:
 \begin{equation}\label{W_n^0}
\wr\{\sin(x),\sin(2x),\sin(3x),\dots,\sin((n-1)x),\sin(nx) \}= (-2)^{\frac{n(n-1)}{2}}\sin^{\frac{n(n+1)}{2}}(x)\cdot G(n+1),
 \end{equation}
where $G(\cdot)$ stands for the Barnes $G$-function and its value at positive integers is given by 
$$ 
G(n+1)=\prod_{j=0}^{n-1}j!
$$ 
To present our main result, which is an extension of \eqref{W_n^0},  we denote $W^{(k)}_1(x):=\sin((k+1)x)$ and
 \begin{equation*}
W^{(k)}_n(x):=\wr\{\sin(x),\sin(2x),\sin(3x),\dots,\sin((n-1)x),\sin((n+k)x) \}, \;\; \; n\ge 2. 
 \end{equation*}
The following theorem constitutes our main result.
\begin{theorem}\label{thm1} 
For integer $k\ge 0$ and $n\ge 1$ 
\begin{equation}\label{eqn:Wkn_thm1}
W^{(k)}_{n}(x)=(-2)^{\frac{n(n-1)}{2}}\sin^{\frac{n(n+1)}{2}}(x)G(n+1)C^{(n)}_{k}(\cos(x)),
 \end{equation}
where $C^{(n)}_{k}(x)$ is the Gegenbauer polynomial.
\end{theorem}

We provide two proofs of Theorem~\ref{thm1}. In Section~\ref{section:proof-1}, we present a direct proof based on recurrence relations. In Section~\ref{proof-2}, we give a more intuitive (though longer) proof using Darboux--Crum transformations. More precisely, the second proof only establishes \eqref{eqn:Wkn_thm1} up to a $\pm$ sign, but it clearly illustrates how and why Gegenbauer polynomials appear in \eqref{eqn:Wkn_thm1}.

\section{Computing Wronskians using recurrence relations}\label{section:proof-1}
For a smooth function $f : \r \mapsto \r$ we define 
\begin{equation}
     {\mathbf V}_n(f(x)):=(f(x), f'(x), f''(x),\dots, f^{(n)}(x))^T \in \r^{n+1}
\end{equation}
We now present the following proposition from \cite{Larsen_1990}:
\begin{proposition}\label{Prop1}
The Wronskian determinant satisfies  
 \begin{equation*}
\wr \{ f, fg_1,...fg_{n-1}\}=f^n \wr \{ g_1,...,g_{n-1}\}
\end{equation*}
where $f,g_1,,\dots,g_{n-1}$ can be any smooth functions.
\end{proposition}
Note that the above formula can be stated in terms of our notation as follows: 
    \begin{equation}\label{Prop1:eqn1}
  \det \big(\mathbf V_n(f),\mathbf V_n(fg_1),\dots \mathbf V_n(fg_{n-1})   \big)=f^n \det( \mathbf V_{n-1}(g'_1),\dots,\mathbf V_{n-1}(g'_{n-1}))
\end{equation}
We use this result to establish the following lemma:
\begin{lemma}\label{Lemma1}
Equation \eqref{eqn:Wkn_thm1} holds for $k=1$, that is 
    \begin{equation}\label{eqn:W1n}
      W^{(1)}_{n}(x)=(-2)^{\frac{n(n-1)}{2}} \sin^{\frac{n(n+1)}{2}}(x) G(n+1)2n \cos(x), 
    \end{equation}
    for all $n\ge 1$. 
\end{lemma}
\begin{proof}  
We recall that the Chebyshev polynomial of the second kind $U_n(x)$, satisfies 
\begin{equation}\label{eq_Un}
\sin ((n+1)x)=\sin(x)U_n(\cos(x)),
\end{equation}
see formula (8.940.2) from \cite{Gradshteyn2007}. To simplify the notation in the rest of the proof, we will denote 
\begin{equation}\label{def_a_b}
a=a(x):=\sin(x), \;\;\;\; b=b(x):=\cos(x). 
\end{equation}
Thus, $W^{(1)}_{n}(x)=\wr \{ \sin(x),\sin(2x),\dots,\sin((n-1)x),\sin((n+1)x) \}$ can be written as
$$
W^{(1)}_{n}(x)=\det \big(\mathbf V_n(a),\mathbf V_n(aU_1(b)),\dots,\mathbf V_n(aU_{n-2}(b)),\mathbf V_n(aU_{n}(b)) \big). 
$$
The Chebyshev polynomial $U_n(b)$ is an $n^\text{th}$-order polynomial containing only terms of the form $b^{n-2m}$, and its leading coefficient is $2^n$. Writing the determinant $W^{(1)}_{n}(x)$ as a sum of determinants whose last column equals $\mathbf V_n(a b^{n-2m})$, we see that each term (except for the one with $m=0$) vanishes because the columns are linearly dependent. Thus, we can simplify the above formula as
\begin{equation}\label{omit}
W^{(1)}_{n}(x)=\det \big( \mathbf V_n(a),\mathbf V_n(a(2b)),\dots,\mathbf V_n(a(2^{n-2}b^{n-2})),\mathbf V_n(a(2^nb^n)) \big).
\end{equation}
Applying Proposition \ref{Prop1:eqn1} we obtain
\begin{align*}
W^{(1)}_{n}(x)&=\det \big( \mathbf V_{n-1}(-2a), \dots,\mathbf V_{n-1}(-2(n-2)a(2^{n-3} b^{n-3})),\mathbf V_{n-1}(-2n a(2^{n-1} b^{n-1}))\big)\\
&=(-2)^{n-1}n(n-2)! a^n W^{(1)}_{n-1}(x) .
\end{align*}
This gives us a recursion formula for $W_n^{(1)}$. Since we know the inital value $W_1^{(1)}=\sin(2x)$, formula \eqref{eqn:W1n} can now be easily obtained by induction.
\end{proof}  
Next, we establish the following recursion relation.  
\begin{lemma}\label{Lemma2}
   For $n\ge 2$ and $k\ge 0$, we have
    \begin{equation}\label{eqn:Wnk_recurrence}
      W^{(k+2)}_{n}(x)=W^{(k)}_{n}(x)+(-2)^{n-1}(n+k+1)(n-2)!\sin^n(x)W^{(k+2)}_{n-1}. 
    \end{equation}
\end{lemma}
\begin{proof}
Using \eqref{eq_Un} and \eqref{def_a_b},we write 
$$
 W^{(k+2)}_{n}(x)=\det \big(\mathbf V_n(a), \mathbf V_n(aU_1(b)),\dots, \mathbf V_n(aU_{n-2}(b)), \mathbf V_n(aU_{n+k+1}(b))\big).
$$
Chebyshev polynomials satisfy the following three-term recurrence relation (see formula (8.940.2) in \cite{Gradshteyn2007}):
$$
U_{n+k+1}(b) = -U_{n+k-1}(b) + 2b\,U_{n+k}(b).
$$
Thus, we can write $W^{(k+2)}_{n}(x) = A + B$ as the sum of two determinants:
$$
A= \det \big( \mathbf V_n(a), \mathbf V_n(a(U_1(b)),\dots, \mathbf V_n(aU_{n-2}(b)), \mathbf V_n(-aU_{n+k-1}(b))\big)=-W_n^{(k)},
$$
and 
\begin{align*}
B&=\det \big( \mathbf V_n(a), \mathbf V_n(a(U_1(b)),\dots, \mathbf V_n(aU_{n-2}(b)), \mathbf V_n(2abU_{n+k}(b))\big)\\
&=\det \big( \mathbf V_n(a), \mathbf V_n(a(U_1(b)),\dots, \mathbf V_n(aU_{n-2}(b)), \mathbf V_n(2b\sin((n+k+1)x) )\big)\\
&=\det \big(\mathbf V_n(a), \mathbf V_n(a(U_1(b)),\dots, \mathbf V_n(aU_{n-2}(b)), \mathbf V_n(2\sin((n+k)x) )\big)\\
&+\det \big( \mathbf V_n(a), \mathbf V_n(a(U_1(b)),\dots, \mathbf V_n(aU_{n-2}(b)), \mathbf V_n(2a\cos((n+k+1)x) )\big) \\
&=2W_n^{(k)}+2\det \big( \mathbf V_n(a), \mathbf V_n(a(U_1(b)),\dots, \mathbf V_n(aU_{n-2}(b)), \mathbf V_n(2a\cos((n+k+1)x) )\big)
\end{align*}
In the second step of simplifying the expression for $B$, we used the identity
$$
\sin((n+k)x)=\sin((n+k+1)\cos(x)-\cos((n+k+1)x)\sin(x)
$$
We now use \eqref{Prop1:eqn1} for the second term of $B$ and obtain
\begin{align*}
B&=2W_n^{(k)}(x)+2a^n \det \big(\mathbf V_{n-1}(\partial_x (U_1(b)),\dots, \mathbf V_{n-1}(\partial_x (U_{n-2}(b)), \mathbf V_{n-1}(\partial_x (\cos((n+k+1)x)) \big)\\
&=2W_n^{(k)}(x)+2a^n \det \big(\mathbf V_{n-1}(\partial_x (2b)),\dots, \mathbf V_{n-1}(\partial_x (2^{n-2}b^{n-2}), \mathbf V_{n-1}(\partial_x (\cos((n+k+1)x)) \big)\\
&=2W_n^{(k)}(x)+2a^n \det \big(\mathbf V_{n-1}(-2a), \dots \\
& \qquad \dots,\mathbf V_{n-1}(-2(n-2)a(2^{n-3}b^{n-3})),\mathbf V_{n-1}(-(n+k+1)\sin((n+k+1)x)) \big)\\
&=2W_n^{(k)}(x)+(-2)^{n-1}(n+k+1)(n-2)! a^n W_{n-1}^{(k+2)}(x)
\end{align*}
In the second step, we applied the same idea as in deriving \eqref{omit} to remove all but the highest-order terms in $U_n(b)$. In the last step, we used the “reverse” idea of adding extra lower-order terms to recover $aU_n(b) = \sin((n+1)x)$ from $ab^n$.
Combining the above equations gives us the desired result:
    \begin{equation*}
      W^{(k+2)}_{n}(x)=A+B=W^{(k)}_{n}(x)+(-2)^{n-1}(n+k+1)(n-2)!a^{n}W^{(k+2)}_{n-1}(x)
    \end{equation*}
\end{proof}  

Now, we define
\begin{equation}\label{def:tildeWnk}
\widetilde W_n^{(k)}(x) := (-2)^{\frac{n(n-1)}{2}} \sin^{\frac{n(n+1)}{2}}(x)\,G(n+1)\,C^{(n)}_{k}\bigl(\cos(x)\bigr).
\end{equation}
Our goal is to show that \(\widetilde W_n^{(k)}\) also satisfies the recursion relation \eqref{eqn:Wnk_recurrence}.

\begin{lemma}\label{Lemma 3}
  For \(n \ge 2\) and \(k \ge 0\),
  \begin{equation}\label{eqn:tildeW_recurrence}
     \widetilde W^{(k+2)}_{n} = \widetilde W^{(k)}_{n} + (-2)^{n-1} (n+k+1)\,(n-2)!\,\sin^n(x)\,\widetilde W^{(k+2)}_{n-1}.
  \end{equation}
\end{lemma}

\begin{proof}
To prove \eqref{eqn:tildeW_recurrence}, it suffices to show that
\begin{equation}\label{eqn:Cnk_recurrence}
C^{(n)}_{k+2}(\cos(x)) = C^{(n)}_{k}(\cos(x))+ \frac{n+k+1}{n-1}\,C^{(n-1)}_{k+2}(\cos(x)).
\end{equation}
According to formulas (8.933.1) and (8.933.2) from \cite{Gradshteyn2007}, the Gegenbauer polynomials satisfy
$$
k\,C_k^{(n)}(\cos(x))= 2n\bigl[\cos(x)\,C_{k-1}^{(n+1)}(\cos(x))- C_{k-2}^{(n+1)}(\cos(x))\bigr]
$$
and
$$
(2n + k)\,C_k^{(n)}(\cos(x))= 2n\bigl[C_{k}^{(n+1)}(\cos(x))- \cos(x)\,C_{k-1}^{(n+1)}(\cos(x))\bigr].
$$
By adding these two equations together and then replacing \(k\) with \(k+2\) and \(n\) with \(n-1\), we get
\begin{equation*}
2\,(n + k + 1)\,C_k^{(n)}(\cos(x))= 2\,(n - 1)\bigl[C_{k+2}^{(n)}(\cos(x))- C_{k}^{(n)}(\cos(x))\bigr],
\end{equation*}
which is precisely \eqref{eqn:Cnk_recurrence}.
\end{proof}

\noindent
{\it Proof of Theorem \ref{eqn:Wkn_thm1}:} 
We need to show that $\widetilde W^{(k)}_{n} = W^{(k)}_n$ for all $n \ge 1$ and $k \ge 0$. This result is clearly true when $n = 1$ (for all $k \ge 0$) and when $k = 0$ (for all $n \ge 1$); see \eqref{W_n^0}. It is also true when $k = 1$ (for all $n \ge 1$) due to \eqref{eqn:W1n}. We claim that these facts, combined with the recurrence relations \eqref{eqn:Wnk_recurrence} and \eqref{eqn:Cnk_recurrence}, imply that \eqref{eqn:Wkn_thm1} holds for all $n \ge 1$ and $k \ge 0$.

\begin{figure}[t]
\centering
\includegraphics[width=10cm, height=8cm]{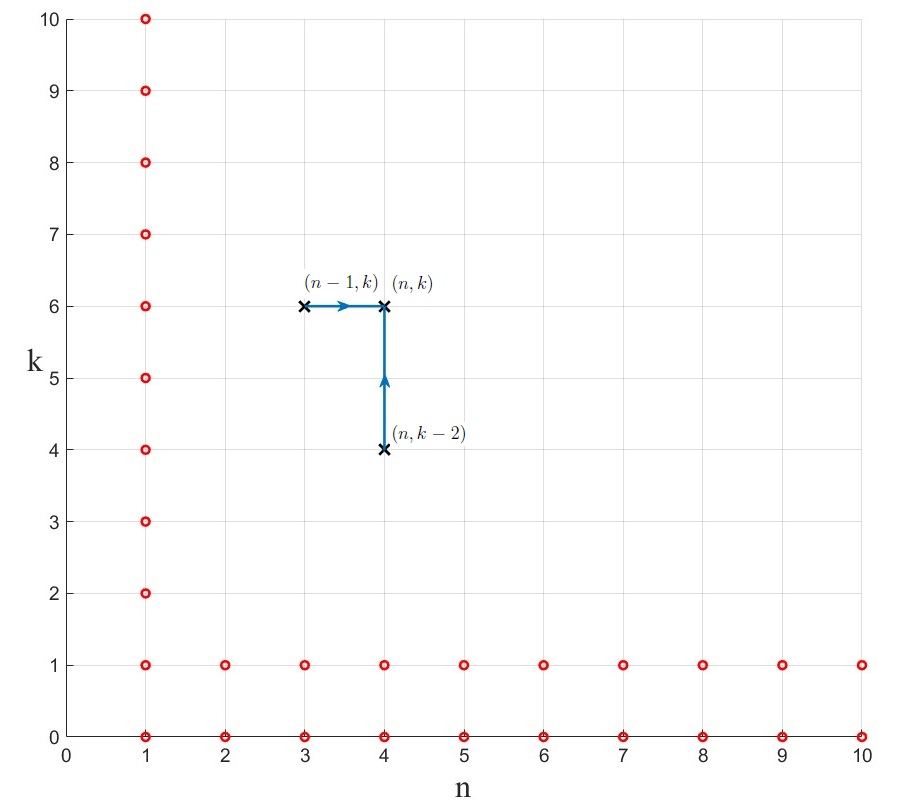}
\caption{Computing $W^{(k)}_{n}$ and $\widetilde 
W^{(k)}_{n}$ via recurrence relations 
\eqref{eqn:Wnk_recurrence} and \eqref{eqn:Cnk_recurrence}.}
\label{fig1}
\end{figure}

The easiest way to prove this is by induction. Given the values of $W^{(k)}_n$ that we already know (indicated by red points at coordinate $(n,k)$ in Figure \ref{fig1}), we can apply the recurrence \eqref{eqn:Wnk_recurrence} and uniquely determine the values of $W^{(k)}_n$ for $k = 3$, then for $k = 4$, and so on. Since $\widetilde W^{(k)}_n$ satisfies the same recurrence relation as $W^{(k)}_n$ and shares the same “initial conditions” (meaning the values for $k = 0, 1$ and $n = 1$), we must have $W^{(k)}_n = \widetilde W^{(k)}_n$ for all $n \ge 1$ and $k \ge 0$.
\qed

In the next section, we will establish the identity \eqref{eqn:Wkn_thm1} (up to a $\pm$ sign) using entirely different ideas. In particular, we will show that the appearance of Gegenbauer polynomials in \eqref{eqn:Wkn_thm1} is not surprising, since the determinants $W^{(k)}_n$ arise when performing a Darboux--Crum transformation and studying the resulting Sturm--Liouville problem.

\section{Computing Wronskians using the Darboux transformation}\label{proof-2}

For $\mu \ge 0$, we consider a second-order linear differential operator
\begin{equation}\label{def:L_mu}
\L_{\mu}:=\frac{1}{2} \partial_x^2 - \mu \csc^2(\pi x),
\end{equation}
acting on twice continuously differentiable functions \(f: (0,1) \to \mathbb{R}\). 
If \(\mu=0\), we impose the Dirichlet boundary conditions \(f(0) = f(1) = 0\). If \(\mu > 0\), boundary conditions are unnecessary, since both boundary points \(0\) and \(1\) are \emph{natural} for the operator \(\L_{\mu}\) in the language of Feller boundary classification 
(see \cite{Borodin2002} [Section II.6]). 
Our first goal is to find a complete eigenbasis \(\{f_k^{(\mu)}\}_{k \ge 0}\) of \(L_2((0,1), \mathrm{d}x)\). When \(\mu=0\), it is well known that the functions
\[
f_k^{(0)}(x) = \sin\bigl(\pi (k+1)\,x\bigr), \quad k \ge 0,
\]
provide an orthogonal basis for functions in \(L_2((0,1), \mathrm{d}x)\) that satisfy the Dirichlet boundary conditions \(f(0)=f(1)=0\).
To determine the eigenbasis for general \(\mu > 0\), we need to recall some results about the Gegenbauer polynomials \(C_k^{(\mu)}(u)\). These polynomials 
can be defined as the solutions to the Gegenbauer differential equation (see \cite{Suetin_2001}).

\begin{align}\label{eqn:ODE_Gegenbauer}
(1-u^2)y''(u)-(2\mu+1)uy'(u)+k(k+2\mu)y(u)=0,
\end{align}
which satisfy 
\begin{equation}\label{eqn:orthogonality_Ckmu}
\int_{-1}^1 C_k^{(\mu)}(t) C_l^{(\mu)}(t)(1-t^2)^{\mu-\frac{1}{2}} \d t=\frac{\pi 2^{1-2\mu}\Gamma(k+2\mu)}{k!(k+\mu)[\Gamma(\mu)]^2} \delta_{k,l}, \;\;\; k,l \ge 0,
\end{equation}
see formula (8.939.8) from \cite{Gradshteyn2007}. The Gegenbauer polynomials form an orthogonal basis of  \\ $L_2((-1,1), (1-t^2)^{\mu-1/2} \d t$. 
Armed with these results, we can state  the following 

\begin{proposition}\label{Prop2}
Denote 
\begin{equation}\label{def_nu}
\nu:=\frac{1}{2} \big( 1 + \sqrt{1+8\mu/\pi^2} \big)
\end{equation}
and
\begin{equation*}
f^{(\mu)}_k(x):=2^{\nu-1}\Gamma(\nu)\sin^{\nu}(\pi x)C_k^{(\nu)}(\cos (\pi x)), \;\;\; k \ge 0. 
\end{equation*}
Then for all $k\ge 0$ and $x \in (0,1)$ 
\begin{equation*}
\L_{\mu} f^{(\mu)}_k(x) = -\frac{1}{2}\pi^2(\nu+k)^2 f^{(\mu)}_k(x).
\end{equation*}
The functions $\{f_k^{(\mu)}\}_{k\ge 0}$ form a complete orthogonal basis of $L_2((0,1), \d x)$ and they satisfy
\begin{equation}\label{eqn:normf}
\int_{0}^1  f^{(\mu)}_k(x)^2 \d x=\frac{\Gamma(2\nu+k)}{2(\nu+k)k!}. 
\end{equation}
\end{proposition}
This result follows from \cite{Dong2014} by change of variables $x=u-1/2$. One could also prove it directly by using the equations \eqref{eqn:ODE_Gegenbauer} and \eqref{eqn:orthogonality_Ckmu}.

%
%
%
%
%

%

\vspace{0.25cm}

\noindent 
{\bf Remark 1.}
The above result allows us to write down a spectral expansion of the transition probability density of  diffusion process $X$, which is defined as a Brownian motion on the interval $(0,1)$ killed  at rate $\mu \csc^2(\pi x)$. When $\mu=0$ and we impose Dirichlet boundary conditions, this process is just the Brownian motion on $(0,1)$ killed at the first time that it hits the boundary of the interval (see page 122 in  \cite{Borodin2002}). When $\mu>0$, it follows from \cite{Borodin2002}[Section II.6] (see also \cite{KY_2024}) that both boundaries are natural for $X$ and its dynamics is defined solely by its Markov infinitesimal generator $\L_{\mu}$. For $\mu>0$ the transition probability density function of $X$ is given by
\begin{equation}\label{eqn:p_txy}
p_t(x,y)=\sum_{k=0}^{+\infty}e^{-\frac{\pi^2 t}{2}(\nu+k)^2}f^{(\mu)}_k(x)f^{(\mu)}_k(y) \frac{2(\nu+k)k!}{\Gamma(2\nu+k)}.
\end{equation}
This result is a generalization  of formulas obtained in \cite{KY_2024}[Section 5.4].  It is also a special case of the general results on diffusion semigroups whose generators have discrete spectrum-see  equation 3.20 in  \cite{Vadim_2007}. The series in \eqref{eqn:p_txy} converges pointwise due to the asymptotic result
$$
C_k^{(\nu)}(\cos(\pi x))=\frac{\sin(\pi x)^{-\nu} 2^{\nu} (2\nu)_n}{ \sqrt{\pi n} (\nu+1/2)_n} \cos((k+\nu)\pi x- \pi \nu/2) + O(n^{-3/2}), \;\;\; n\to +\infty. 
$$
which can be derived  from formulas 
18.7.1	and 18.15.4\_5 from \cite{NIST_2010}.

%
%
%

\vspace{0.25cm}

We next introduce the concept of the Darboux transformation. We start with a second order linear differential operator 
 \begin{equation*}
 {\mathcal L}=\frac{1}{2} \partial_x^2 - c(x), \;\;\; x \in (l,r),
 \end{equation*}
 and a function $\phi: (l,r)\mapsto (0,\infty)$ which satisfies the eigenvalue equation  ${\mathcal L} \phi = \lambda \phi$. We define a first order differential operator 
\begin{equation}\label{def:D_phi}
{\mathcal D}_{\phi}:=\partial_x-\frac{\phi'(x)}{\phi(x)}=\phi(x) \partial_x \Big(\frac{1}{\phi(x)}\Big)
\end{equation}
 and check that
\begin{equation*}
{\mathcal L}=\lambda {\mathcal I} + \frac{1}{2}{\mathcal A} {\mathcal B},
\end{equation*}
where
\begin{equation*}
{\mathcal A}={\mathcal D}_{\frac{1}{\phi}}=\partial_x+\frac{\phi'(x)}{\phi(x)}, \;\;\;
{\mathcal B}={\mathcal D}_{\phi}:=\partial_x-\frac{\phi'(x)}{\phi(x)}.
\end{equation*}
The Darboux transform of ${\mathcal L}$ is defined as
\begin{equation*}
 \widetilde {\mathcal L}=\lambda {\mathcal I} + \frac{1}{2} {\mathcal B} 
 {\mathcal A}=\frac{1}{2} \partial_x^2 - \tilde c(x),
\end{equation*}
where 
\begin{equation*}
\tilde c(x):=c(x)-\partial_x^2 \ln \phi(x). 
\end{equation*}
The function $\phi$ that is used to construct the Darboux transform is called {\it the seed function}.
One of the important property of the Darboux transformation, is the following intertwining relation 
\begin{equation}\label{L_tilde_L_intertwining}
 \widetilde \L \D_\phi = \D_\phi \L,
\end{equation}
In particular, if $f$ is a solution to $\L f=\kappa f$ then $g=\D_\phi f$ solves $\widetilde \L g=\kappa g$.

The Darboux transform can be iterated, in which case it is called 
called the Darboux-Crum transformation.  The main properties of the Darboux -Crum transformation are collected in the next 

\begin{theorem}[\cite{Crum_1955,Gomez_2020}]\label{thm_Crum}
Let $h_1,\dots,h_n$ be a set of $n$ eigenfunctions of $\mathcal L$. We perform an $n$-step Darboux transformation with these seed eigenfunctions and obtain a chain of second order linear  operators
$$
\mathcal L\to \mathcal L^{(1)}\to \dots \to \mathcal L^{(n)}.
$$
The operator $\mathcal L_n$ is given by
\begin{equation}\label{eqn:L_n}
\mathcal L^{(n)}=\frac{1}{2}\partial^2_{x}- c_n(x)=\frac{1}{2}\partial^2_{x}-c(x)+ 
\partial_x^2 \ln \textnormal{Wr}\{h_1,\dots,h_n \}
\end{equation}
If $\psi$ is a solution to $\mathcal L \psi = \kappa \psi$, then the function
\begin{equation}\label{Crum}
\psi_n=\frac{\textnormal{Wr}\{h_1,\dots,h_n,\psi \}}{\textnormal{Wr}\{h_1,\dots,h_n \}}
\end{equation}
is a solution to $\mathcal L^{(n)} \psi_n = \kappa \psi_n$.
\end{theorem}


Consider an operator $\L=\frac{1}{2}\partial^2_{x}$ on $x \in (0,1)$ with Dirichlet boundary conditions. It is clear that $h_k(x)=\sin(\pi k x)$ are the eigenfunctions of $\L$ with eigenvalues $-\pi^2 k^2/2$. We perform Darboux-Crum transformation times with $h_k(x)=\sin(k \pi x)$, $k=1,2,\dots,n-1$. According to \eqref{W_n^0}, we have  $\textnormal{Wr}\{h_1,\dots,h_{n-1} \}=C(n) \sin^{\frac{n(n-1)}{2}}(\pi x)$ and from \eqref{eqn:L_n} we conclude that  the transformed operator $\L^{(n-1)}$ is 
$$
\L^{(n-1)}=\frac{1}{2} \partial_x^2-\frac{1}{2}n(n-1)\pi^2\csc^2(\pi x). 
$$

Next we take $\psi(x)=h_{n+k}(x)=\sin((n+k)\pi x)$ and compute $\psi_{n-1}$ in formula \eqref{Crum} (which we denote by $g_k^{(n-1)}(x)$)
\begin{equation}\label{eqn:g}
g^{(n-1)}_k(x)=\frac{\wr\{\sin(\pi x), \sin(2\pi x), \dots, \sin((n-1)\pi x),
\sin((n+k) \pi x) \}}{\wr\{\sin(\pi x), \sin(2\pi x), \dots, \sin((n-1)\pi x)\}}=\frac{W_{n}^{k}(\pi x)}{W_{n-1}^{0}(\pi x)}. 
\end{equation}
According to Theorem \ref{thm_Crum}, the function $g^{(n-1)}_k(x)$ 
is a solution to the second-order linear differential equation $\L^{(n-1)} g = - \frac{1}{2}\pi^2 (n+k)^2 g$. 
On the other hand, we note that $\L^{(n-1)}=\L_{\mu_n}$ as defined in \eqref{def:L_mu}, with
\begin{equation}\label{eqn:mu}
\mu_n:=\frac{1}{2} n(n-1) \pi^2.  
\end{equation} 
From \eqref{def_nu} we find that $\nu=n$, and Proposition \ref{Prop2} tells us that $f_{k}^{(\mu_n)}(x)$ is also a solution to $\L^{(n-1)} g = - \frac{1}{2}\pi^2 (n+k)^2 g$. Because a second-order linear differential equation $\L^{(n-1)} g = - \frac{1}{2}\pi^2 (n+k)^2 g$ has two linearly independent solutions, and Proposition \ref{Prop2} indicates that $f_{k}^{(\mu_n)}(x)$ is the unique solution (up to multiplication by a constant) in $L_2((0,1),\d x)$, it follows from Crum \cite{Crum_1955} that the solutions constructed in Theorem \ref{thm_Crum} are $L_2((0,1),\d x)$ closed and complete, and converge to zero as $x\to 0^+$ or $x\to 1-$.  We conclude that the two solutions $g^{(n-1)}_k(x)$ and $f^{(\mu_{n})}_{k}(x)$ to $\L^{(n-1)} g = - \frac{1}{2}\pi^2 (n+k)^2 g$ differ only by a multiplicative constant:
\begin{equation}\label{eqn:fg}
g^{(n-1)}_k(x)=C f^{(\mu_{n})}_{k}(x),\ \ \ k\ge 0. 
\end{equation}

We can determine the constant $C$ (up to $\pm$ sign)  via the following result, whose proof can be found in \cite{KY_2024}. 
 \begin{proposition}\label{Prop3}
Consider a second order linear differential operator $\L=\frac{1}{2} \partial_y^2-c(y)$, where $c$ is a continuous function of $y\in (l,r)$.  Assume that $h$, $f$ and $g$ are  $C^2$ functions on $(l,r)$ such that 
\begin{itemize}
\item[(i)] $h$ is positive and satisfies $\L h=\lambda h$ for $y\in (l,r)$;
\item[(ii)] $g$ satisfies $\L g=\kappa g$ for $y\in (l,r)$. 
\end{itemize}
Denote $\tilde f={\mathcal D}_h f$ and $\tilde g=\mathcal D_h g$. Where $\D$ is define in \eqref{def:D_phi}. Then  
\begin{equation}\label{Prop3:eqn2}
\int \tilde f(y) \tilde g(y) \d y=f(y) \tilde g(y) +2(\lambda-\kappa) \int f(y) g(y) \d y.
\end{equation}
\end{proposition}
We apply Proposition \ref{Prop3} with $\L=\L^{(n-1)}$,  $g=f=g^{(n-1)}_{k}(x)$,  $\kappa=- \frac{1}{2}\pi^2 (n+k)^2$, and $\lambda=- \frac{1}{2} n^2 \pi^2$.  From \eqref{Prop3:eqn2} we find
$$
\int {\tilde g}^2(x)  \d x=g(x) \tilde g(x) +((n+k)^2-n^2)\pi^2\int g^2(x)\d x. 
$$
The intertwining relationship \eqref{L_tilde_L_intertwining} implies 
$\tilde g={\mathcal D}_{h} g^{(n-1)}_{k}=g^{(n)}_{k-1}$. According to Crum \cite{Crum_1955}, both function $g(x)$ and $\tilde g(x)$ approach zero as $x\to 0^+$ or $x\to 1-$. Thus 
$$
|| \tilde g ||^2=\int_0^{1} {\tilde g}(x)^2  \d x=g(1-) \tilde g(1-) -g(0+) \tilde g(0+)+((n+k)^2-n^2)\pi^2\int_0^{1} {g}(x)^2  \d x=((n+k)^2-n^2)\pi^2 || g ||^2,
$$
and we obtain the recurrence relation:
$$
||g^{(n)}_{k-1}||^2=\pi^2((n+k)^2-n^2)||g^{(n-1)}_{k}||^2. 
$$
Since $|| g^{(0)}_k||=\int_0^{1} \sin^2(\pi kx)  \d x=1/2$, we find the $L_2((0,1),\d x)$ norm of $g^{(n-1)}_n$ as follows
$$
||g^{(n-1)}_k||^2=\pi^{2n-2}\prod_{i=1}^{n-1}((n+k)^2-i^2)||g^{(0)}_{n+k-1}||^2=\frac{1}{2}\pi^{2n-2}\prod_{i=1}^{n-1}((n+k)^2-i^2)=\frac{1}{2}\pi^{2n-2}\frac{\Gamma(2n+k)}{(n+k)k!}. 
$$
From \eqref{eqn:normf} we find
\begin{equation}
||f^{(\mu_{n})}_{k}(x)||^2=\int_{0}^1 (f^{(\mu_{n})}_{k}(x))^2  \d x=\frac{\Gamma(2n+k)}{2(n+k)k!}.
\end{equation}
Combining the above two results with \eqref{eqn:fg}, shows that 
$C=\pm \pi^{n-1}$. Hence 
$$
\frac{W_{n}^{k}(\pi x)}{W_{n-1}^{0}(\pi x)}=g^{(n-1)}_{k}(x)=C  f^{(\mu_{n})}_{k}(x)=\pm \pi^{n-1} 2^{n-1}\Gamma(n) \sin^{n}(\pi x)C_{k}^{(n)}(\cos (\pi x)). 
$$
Since the value of $W_{n-1}^{0}(\pi x)$ is known from \cite{Larsen_1990} (see \eqref{W_n^0}), we conclude that 
$$
W^{(k)}_{n}(\pi x)=W_{n-1}^{0}(\pi x)g^{(n-1)}_{k}(\pi x)=\pm (-2)^{\frac{n(n-1)}{2}}\pi^{\frac{n(n-1)}{2}}G(n+1)\sin^{\frac{n(n+1)}{2}}(\pi x)C_{k}^{(n)}(\cos (\pi x)), 
$$
which reproduces the result from Theorem~\ref{thm1}, only up to a $\pm$ sign.

\end{document}